\documentclass[10pt]{amsart}
\usepackage[utf8]{inputenc}

\usepackage{amssymb,amsthm,amsmath}
\usepackage{mathrsfs}
\usepackage{enumerate}
\usepackage{graphicx,xcolor}
\usepackage{setspace}
\usepackage[hidelinks]{hyperref}
\usepackage{comment}
\usepackage{bbm}

\usepackage{pgfplots}
\usepgfplotslibrary{fillbetween}
\usetikzlibrary{calc}
\pgfplotsset{compat=1.18}

\usepackage[paper=a4paper, left=1.2in, right=1.2in, top=1.2in, bottom=1.2in]{geometry}
\usepackage{dsfont}

\numberwithin{equation}{section} 

\newtheorem{theorem}{Theorem}[section] 
\newtheorem{lemma}[theorem]{Lemma}
\newtheorem{corollary}[theorem]{Corollary}
\newtheorem{proposition}[theorem]{Proposition}

\theoremstyle{remark}
\newtheorem{remark}[theorem]{Remark}

\theoremstyle{definition}

\hypersetup{
    colorlinks,
    linkcolor={teal!94!black},
    citecolor={orange!90!black},
    urlcolor={blue!75!black}
}

\title{Topological Genericity and Large Linear Structures in Function Spaces}

\author{Christos Pandis}

\address{Department of Mathematics \& Applied Mathematics, University of Crete, Voutes Campus, 70013 Heraklion, Greece}

\email{chrpandis@gmail.com}

\date{\today}

\begin{document}
\maketitle
\begin{abstract}
We prove that the set of integrable functions on the unit circle for which the analogue of Paley's theorem for \(H^1\) fails  is residual in \(L^1(\mathbb T)\). Moreover, we establish algebraic genericity and spaceability results in several Hardy-type function spaces under prescribed conditions on Taylor coefficients, extending phenomena considered in \cite{pandis2024some,NestoridisGenericH1}.
\end{abstract}

\vspace{0.5em}
\smallskip\\

\maketitle

\section{Introduction and Results}

A common phenomenon in analysis is that, once the existence of an object with a
certain pathological property has been established, one may ask whether such
objects are abundant. Abundance can be understood in several different senses:
cardinality, measure, Baire category, or the existence of large linear
structures. A simple guiding example is given by the irrational numbers. In
comparison with the rational numbers, the irrational numbers may be regarded as
the ``generic'' real numbers. Indeed, from the cardinal point of view, the
rationals are countable, whereas the irrationals are uncountable. From the
measure-theoretic point of view, the rationals have Lebesgue measure zero,
whereas the irrationals have full measure. From the topological point of view,
the rationals are meagre in \(\mathbb R\), since they are a countable union of
closed sets with empty interior, while the irrationals form a dense
\(G_\delta\) subset of \(\mathbb R\).

Recall that a subset of a topological space is called \(G_\delta\) if it is a
countable intersection of open sets. A subset is called residual if its
complement is of first category. Equivalently, in a Baire space, a residual set
contains a dense \(G_\delta\) set; see~\cite{oxtoby1980measure}. Such sets are
considered topologically large; thus, a property which holds on a residual set is
often said to hold generically.

Over the last decades, there has been a growing interest in finding large
algebraic and topological structures inside sets which are themselves highly
nonlinear; see, for instance,
\cite{axarlis2023topological,biehler2020algebraic,konidas2026double,nestoridis2022generic,bernal2023bloch,galanos2021genericity,bernal2014linear}
and the references therein.

Let \(E\) be a topological vector space and let \(M\subset E\).
\begin{itemize}
    \item We say that \(M\) is \emph{lineable} if \(M\cup\{0\}\) contains an
    infinite-dimensional vector subspace of \(E\).

    \item We say that \(M\) is \emph{dense-lineable}, or
    \emph{algebraically generic}, if \(M\cup\{0\}\) contains a dense vector
    subspace of \(E\).

    \item Finally, \(M\) is said to be \emph{spaceable} if \(M\cup\{0\}\)
    contains a closed infinite-dimensional vector subspace of \(E\).
\end{itemize}
Given the extensive literature on related topics, we refer the interested reader
to the comprehensive book~\cite{aron2015lineability} for a detailed historical
account and a survey of the current state of the art.

The purpose of this paper is to study these notions for several exceptional sets
defined through coefficient conditions. Our starting point is a collection of
genericity results in~\cite{pandis2024some} and~\cite{NestoridisGenericH1}
concerning Taylor coefficients in Hardy-type spaces. In particular,
\cite{pandis2024some} proves that, topologically generically in
\(\bigcap_{0<p<1}H^p\), the Taylor coefficient sequence is unbounded, and that
analogous statements hold for the Taylor coefficients of derivatives, for
localized Hardy spaces, and for the disc algebra \(A(\mathbb D)\), where the
relevant exceptional condition is the failure of absolute summability of the
Taylor coefficients. We also consider a topological genericity result of
Nestoridis~\cite{NestoridisGenericH1} concerning \(H^1\): if \(f\in H^1\) and
\(F(f)\) denotes its primitive, then Hardy's inequality implies that the Taylor
coefficients of \(F(f)\) belong to \(\ell^1\), whereas Nestoridis showed that,
generically, these coefficients do not belong to any \(\ell^p\) with
\(0<p<1\).

Our aim is to strengthen these topological genericity results in the direction of
linear structure. We prove algebraic genericity and spaceability for the
corresponding coefficient-exceptional sets. In addition, we prove a residuality
result in \(L^1(\mathbb T)\) showing that the analogue of Paley's theorem for
\(H^1\) fails generically.

\subsection{A generic result for lacunary Fourier coefficients in \(L^1(\mathbb T)\)}
A well-known theorem of Paley states that if
\[
f(z)=\sum_{n=0}^{\infty} a_n z^n \in H^1,
\]
then, for every lacunary sequence \((n_k)_{k\geq 1}\) of positive integers,
that is, for every sequence satisfying
\[
\frac{n_{k+1}}{n_k}\geq Q>1,
\qquad k\geq 1,
\]
one has
\[
\sum_{k=1}^{\infty} |a_{n_k}|^2 < +\infty.
\]
This conclusion need not hold in \(L^1\), even though \(H^1\) is, in many
contexts, a natural substitute for \(L^1\).Moreover, the set of functions for
which this failure occurs is large. We define
\[
\mathcal Z
:=
\left\{
f\in L^1(\mathbb T):
\text{there exists a lacunary sequence } (n_k)_{k\geq 1}
\text{ such that }
(\widehat f(n_k))_{k\geq 1}\notin \ell^2
\right\}.
\]
\begin{theorem}\label{Z residual}
The set \(\mathcal Z\) is residual in \(L^1(\mathbb T)\).
\end{theorem}

To this end, we consider the fixed dyadic lacunary sequence \((2^k)_{k\geq 1}\)
and define
\[
\mathcal Z_2
:=
\left\{
f\in L^1(\mathbb T):
\sum_{k=1}^{\infty} |\widehat f(2^k)|^2 = +\infty
\right\}.
\]

\begin{proposition}\label{Z_2 generic}
The set \(\mathcal Z_2\) is a dense \(G_\delta\) subset of \(L^1(\mathbb T)\).
\end{proposition}

\begin{proof}[Proof of Theorem~\ref{Z residual}]
The dyadic sequence
\[
n_k = 2^k, \qquad k\geq 1,
\]
is lacunary. Therefore,
\[
\mathcal Z_2 \subset \mathcal Z.
\]
By Proposition~\ref{Z_2 generic}, the set \(\mathcal Z_2\) is non-empty and a dense
\(G_\delta\) subset of \(L^1(\mathbb T)\). Hence \(\mathcal Z\) contains a
dense \(G_\delta\) subset of \(L^1(\mathbb T)\) and consequently, \(\mathcal Z\)
is residual in \(L^1(\mathbb T)\).
\end{proof}

\subsection{Topological/Algebraic genericity and Spaceability in Hardy spaces under conditions on Taylor coefficients}
\par\vspace{0.5em}\noindent

For a holomorphic function $f(z)=\sum_{n=0}^{\infty}\beta_n(f)z^n$ on the unit disc we define $\beta(f)$  to be the sequence of Taylor coefficients, thus we write: $\beta(f)=\{\beta_n(f) \}_{n=0}^{\infty}$. Let \(H(\mathbb D)\) denote the space of holomorphic
functions on \(\mathbb D\), endowed with the topology of uniform convergence on
compact subsets.

\begin{theorem}\label{alg spac criter}
Let \(X\) be a non-trivial metrizable topological vector space continuously embedded in \(H(\mathbb D)\), whose topology is induced by a translation-invariant metric. Let
\(0<p\leq \infty\), and set
\[
\mathcal A_p(X)
=
\{f\in X:\beta(f)\notin \ell^p\}.
\]
Assume that:
\begin{enumerate}
    \item the set \(\mathcal A_p(X)\) is nonempty;
    \item the polynomials are dense in \(X\);
    \item if \(f\in X\), then the functions \(f(-z)\), \(zf(z)\), and
    \(f(z^N)\), \(N\in\mathbb N\), also belong to \(X\).
\end{enumerate}
Then \(\mathcal A_p(X)\) is topologically generic, algebraically generic, and
spaceable in \(X\).
\end{theorem}

\begin{remark}
For the topological genericity conclusion, only assumptions \((1)\) and \((2)\)
are needed.
\end{remark}

\begin{corollary}\label{cor:A-infty-A1}
The sets $\mathcal A_\infty\left(\bigcap_{0<p<1}H^p\right)$ and $\mathcal A_1\bigl(A(\mathbb D)\bigr),$ introduced in~\cite{pandis2024some}, are topologically generic,
algebraically generic, and spaceable in \(\bigcap_{0<p<1}H^p\) and
\(A(\mathbb D)\), respectively. In particular, the topological genericity
conclusions recover the corresponding results of~\cite{pandis2024some}. 
\end{corollary}

\begin{remark}
The same argument applies to \(H^p\), \(0<p<1\), for the set $\mathcal A_\infty(H^p)$
because \(H^p\) satisfies the same structural assumptions and the function
\[
g(z)=\frac{1}{1-z}\log\left(\frac{1}{1-z}\right)=\sum_{n=1}^{\infty}H_n z^n
\]
belongs to \(H^p\), for $0<p<1$. Where $H_n:=\sum_{j=1}^n\frac{1}{j}$ is the $n$-th harmonic sum.

\end{remark}

With minor modifications in the proof of Theorem~\ref{alg spac criter}, one can
prove the following result.

For \(n\in\mathbb N_0\), the set
\[
\mathcal C_n
:=
\left\{
f\in \bigcap_{0<p<1} H^p:
\beta(f^{(n)})\notin \ell^\infty
\right\},
\]
where \(f^{(0)}=f\), was introduced in~\cite{pandis2024some}, where its
topological genericity was also proved. Motivated by the study of algebraic
genericity in difference sets, see for instance
\cite{bernal2026topological,axarlis2023topological,konidas2026double}, we prove
the following result.

\begin{theorem}\label{C_n-C_n-1}
For every \(n\in\mathbb N\), the set
\[
\mathcal D_n
:=
\mathcal C_n\setminus \mathcal C_{n-1}
=
\left\{
f\in \bigcap_{0<p<1} H^p:
\beta(f^{(n)})\notin \ell^\infty
\ \text{and}\
\beta(f^{(n-1)})\in \ell^\infty
\right\}
\]
is algebraically generic in \(\bigcap_{0<p<1} H^p\).
\end{theorem}

\begin{remark}
The same result holds if \(\bigcap_{0<p<1}H^p\) is replaced by \(H^p\), for any
fixed \(0<p<1\).
\end{remark}

Finally, let \(f\in H^1\), and let $F(f)(z):=\int_0^z f(\zeta)\,d\zeta$ be its primitive vanishing at the origin. We write $a(f):=(a_n(f))_{n=0}^{\infty}$ for the Taylor coefficients of \(F(f)\). For \(0<p<1\), Nestoridis~\cite{NestoridisGenericH1}
showed that the set
\[
\Lambda_p
:=
\left\{
f\in H^1:a(f)\notin \ell^p
\right\}
\]
is dense \(G_\delta\) in \(H^1\). Moreover, Baire's theorem guarantees that
\[
\Lambda
:=
\bigcap_{0<p<1}\Lambda_p
=
\left\{
f\in H^1:a(f)\notin \ell^p \ \text{for all } 0<p<1
\right\}
\]
is dense \(G_\delta\) in \(H^1\) and, in particular, nonempty.

\begin{theorem}\label{alg spac nestor}
The set \(\Lambda\) is algebraically generic and spaceable in \(H^1\).
\end{theorem}

Since \(\Lambda \subseteq \Lambda_p\) for every \(0<p<1\), one also obtains the
following.

\begin{corollary}\label{res for lamba_p}
For every \(0<p<1\), the set \(\Lambda_p\) is algebraically generic and
spaceable in \(H^1\).
\end{corollary}
\section{Preliminaries}\label{sec:2}

Let $ \mathbb D=\{z\in\mathbb C:|z|<1\}$
be the open unit disc, and let \(H(\mathbb D)\) denote the space of holomorphic
functions on \(\mathbb D\), endowed with the topology of uniform convergence on
compact subsets. A holomorphic function \(f:\mathbb D\to\mathbb C\) belongs to
the Hardy space \(H^p\), \(0<p<+\infty\), if
\[
\sup_{0<r<1}
\frac{1}{2\pi}\int_0^{2\pi}|f(re^{i\theta})|^p\,d\theta
<+\infty.
\]
It belongs to \(H^\infty\) if $\sup_{|z|<1}|f(z)|<+\infty.$
 
The space \(H^\infty\), endowed with the supremum norm on \(\mathbb D\), is a
Banach space, but polynomials are not dense in this space. For
\(1\leq p<+\infty\), the space \(H^p\), endowed with the norm
\[
\|f\|_p
=
\sup_{0<r<1}
\left\{
\frac{1}{2\pi}\int_0^{2\pi}|f(re^{i\theta})|^p\,d\theta
\right\}^{1/p},
\]
is also a Banach space. Thus, for \(f,g\in H^p\), their distance is given by
\[
d_p(f,g)
=
\sup_{0<r<1}
\left\{
\frac{1}{2\pi}\int_0^{2\pi}
|f(re^{i\theta})-g(re^{i\theta})|^p\,d\theta
\right\}^{1/p},
\qquad 1\leq p<+\infty.
\]

For \(0<p<1\), we endow \(H^p\) with the metric
\[
d_p(f,g)
=
\sup_{0<r<1}
\frac{1}{2\pi}\int_0^{2\pi}
|f(re^{i\theta})-g(re^{i\theta})|^p\,d\theta,
\qquad f,g\in H^p.
\]
With this metric, \(H^p\) becomes a complete metrizable topological vector space,
whose metric is translation-invariant; in other words, \(H^p\) is an \(F\)-space.
For \(0<p<+\infty\), polynomials are dense in \(H^p\). Moreover, convergence in
\(H^p\), \(0<p\leq+\infty\), implies uniform convergence on compact subsets of
\(\mathbb D\).

For \(a,b\in(0,+\infty]\) with \(a<b\), we have \(H^b\subset H^a\), and the
inclusion map is continuous. Jensen's inequality implies that the map
\[
a\mapsto
\sup_{0<r<1}
\left\{
\frac{1}{2\pi}\int_0^{2\pi}|f(re^{i\theta})|^a\,d\theta
\right\}^{1/a}
\]
is increasing. Clearly, we also have
\[
\sup_{0<r<1}
\left\{
\frac{1}{2\pi}\int_0^{2\pi}|f(re^{i\theta})|^p\,d\theta
\right\}^{1/p}
\leq
\sup_{|z|<1}|f(z)|.
\]

We recall Littlewood's subordination principle.

\begin{theorem}[\cite{duren1970theory}, Theorem~1.7]\label{subord}
Let \(f\) and \(F\) be analytic in \(\mathbb D\), and suppose that
\(f(z)=F(w(z))\), where \(w\) is analytic in \(\mathbb D\) and satisfies
\(|w(z)|\leq |z|\). Then
\[
\frac{1}{2\pi}\int_{0}^{2\pi}|f(re^{i\theta})|^p\, d\theta\leq \frac{1}{2\pi}\int_{0}^{2\pi}|F(re^{i\theta})|^p\, d\theta,
\]
for \(0<p<\infty\).
\end{theorem}

We shall also use the following elementary consequence of Hardy's inequality.

\begin{lemma}\label{f' H^1 f H^{infty}}
If \(f\in H^1\), then \(F(f)\in H^\infty\).
\end{lemma}

\begin{proof}
Since \(f\in H^1\), Hardy's inequality~\cite{duren1970theory} implies that the
primitive \(F(f)\) has an absolutely convergent Taylor series on the closed unit
disc. Therefore, \(F(f)\in H^\infty\).
\end{proof}

Next, for \(0<a\leq+\infty\), we consider the intersection $\bigcap_{0<p<a}H^p.$
 
Convergence in this space is equivalent to convergence in \(H^p\) for every
\(p<a\). Equivalently, if \((p_n)\) is a strictly increasing sequence converging
to \(a\), then the metric on \(\bigcap_{0<p<a}H^p\) is given by
\[
d(f,g)
=
\sum_{n=1}^{\infty}
\frac{1}{2^n}
\frac{d_{p_n}(f,g)}{1+d_{p_n}(f,g)},
\qquad
f,g\in\bigcap_{0<p<a}H^p.
\]
This space is also complete; in fact, it is an \(F\)-space. Clearly, convergence
in \(\bigcap_{0<p<a}H^p\) implies uniform convergence on compact subsets of
\(\mathbb D\).

\begin{proposition}\label{prop2.1}
Polynomials are dense in \(\bigcap_{0<p<a}H^p\), for every
\(0<a\leq+\infty\).
\end{proposition}

\begin{proof}
For the proof it suffices for $f\in\bigcap\limits_{p<a}H^p$ to control $d_{p_n}(f,P)$, for $n=1,\ldots,N$ for any finite $N$. Because of the monotonicity of the map $p\to\sup_{0<r<1}\Big\{\dfrac{1}{2\pi}\int^{2\pi}_0|f(re^{i\theta})-P(re^{i\theta})|^pd\theta\Big\}^{1/p}$
is suffices to control $d_{p_N}(f,P)$. But this is possible, because polynomials are dense in $H^{p_N}$ since $p_N<+\infty$. Thus Proposition
\ref{prop2.1} holds.

\end{proof}

The following theorem of Nestoridis and Thirios generalizes
Theorem~5.12 of~\cite{duren1970theory}.

\begin{theorem}[Nestoridis--Thirios~\cite{nestoridis2022generic}]
\label{Nest-Thirios primitive}
Let \(0<p<1\) and let $f\in \bigcap_{0<\beta<p}H^\beta.$ Then $F(f)\in \bigcap_{0<\gamma<q}H^\gamma,$ for $q=\frac{p}{1-p}$. Moreover, if $f\in \bigcap_{0<\beta<1}H^\beta,$ then $F(f)\in \bigcap_{0<\gamma<\infty}H^\gamma.$
 
\end{theorem}

A function \(f:\overline{\mathbb D}\to\mathbb C\) belongs to the disc algebra
\(A(\mathbb D)\) if it is holomorphic on \(\mathbb D\) and continuous on
\(\overline{\mathbb D}\). Thus, $A(\mathbb D)=H^\infty(\mathbb D)\cap C(\overline{\mathbb D}).$ Endowed with the uniform norm, \(A(\mathbb D)\) is a Banach space; in fact, it
is a Banach algebra. Clearly, convergence in \(A(\mathbb D)\) implies uniform
convergence on \(\overline{\mathbb D}\), and hence uniform convergence on compact
subsets of \(\mathbb D\). Polynomials are dense in \(A(\mathbb D)\).

Let $\mathbb T=\{z\in\mathbb C:|z|=1\}$ be the unit circle. The space
\(L^1(\mathbb T)\) is endowed with the norm
\[
\|f\|_{L^1(\mathbb{T})}
=
\frac{1}{2\pi}\int_0^{2\pi}|f(e^{i\theta})|\,d\theta,
\]
which induces its usual topology. For \(f\in L^1(\mathbb T)\), we write
\[
\widehat f(n)
=
\frac{1}{2\pi}\int_0^{2\pi} f(e^{i\theta})e^{-in\theta}\,d\theta,
\qquad n\in\mathbb Z,
\]
for the Fourier coefficients of \(f\). A trigonometric polynomial is a finite
linear combination of the functions \(e_n(\theta)=e^{in\theta}\), \(n\in\mathbb Z\);
that is, it is a function of the form
\[
P(x)
=
\sum_{|n|\leq N} c_n e^{inx},
\qquad c_n\in\mathbb C.
\]
It is well known that the trigonometric polynomials, denoted by
\(\mathcal P_{\mathrm{trig}}\), are dense in \(L^1(\mathbb T)\).

We recall the following well-known theorem.
\begin{theorem}[\cite{katznelson2004introduction}, Theorem~4.1]
\label{Duren thm 4.5}
Let \((a_n)_{n\in\mathbb Z}\) be an even sequence of nonnegative numbers tending
to zero at infinity. Assume that \((a_n)\) is convex, in the sense that
\[
a_{n-1}+a_{n+1}-2a_n \geq  0,
\qquad n> 0.
\]
Then there exists a nonnegative function \(f\in L^1(\mathbb T)\) such that $\widehat f(n)=a_n$.

\end{theorem}

\section{Proofs}

We proceed showing that $\mathcal{Z}_2$, and thus $\mathcal{Z}$, is non-empty.
\begin{lemma}\label{Duren lacunary example}
Let $0<a\leq \frac12$ and define
\[
a_n=\frac1{(\log(|n|+2))^a},
\qquad n \in \mathbb Z.
\]
Then there exist $f\in L^1(\mathbb T)$
 such that $\widehat f(n)=a_n$, but
\[
\sum_{k=1}^{\infty}|\widehat f(2^k)|^2=+\infty.
\]
\end{lemma}

\begin{proof}
We first check that $(a_n)$ is convex, and tends to zero. Consider
\[
\phi(x)=\frac1{(\log(x+2))^a},
\qquad x\geq0.
\]
Then,
\[
\phi''(x)
=
\frac{a}{(x+2)^2}
\left(
\frac1{(\log(x+2))^{a+1}}
+
\frac{a+1}{(\log(x+2))^{a+2}}
\right)>0.
\]
Thus $\phi$ is convex, and therefore the sequence $a_n=\phi(n)$ is convex. Clearly, $a_n\to0.$

By Theorem~\ref{Duren thm 4.5} there exist a nonnegative $f\in L^1(\mathbb T)$ such that $\widehat f(n)=a_n$ and  for this function, we have $\widehat f(n)=\widehat f(-n)$, for
 $n\geq1.$
Therefore
\[
\sum_{k=1}^{\infty}|\widehat f(2^k)|^2
=
\sum_{k=1}^{\infty}a_{2^k}^2
=
\sum_{k=1}^{\infty}
\frac1{(\log(2^k+2))^{2a}}.
\]
Since
\[
\log(2^k+2)\sim k\log2
\qquad (k\to\infty),
\]
we get that 
\[
\sum_{k=1}^{\infty}
\frac1{(\log(2^k+2))^{2a}},
  \quad 
\sum_{k=1}^{\infty}\frac1{k^{2a}}.
\]
have the same behavior.
Because $0<a\leq\frac12,$
we have
\[
\sum_{k=1}^{\infty}|\widehat f(2^k)|^2=+\infty.
\]
\end{proof}

\begin{proof}[Proof of Proposition~\ref{Z_2 generic}]
For \(N\in\mathbb N\), define
\[
S_N(f)
=
\sum_{k=1}^{N}|\widehat f(2^k)|^2,
\qquad f\in L^1(\mathbb T).
\]
In order to show that \(\mathcal Z_2\) is a \(G_\delta\) set, it suffices to
show that \(L^1(\mathbb T)\setminus \mathcal Z_2\) is a countable union of
closed subsets of \(L^1(\mathbb T)\). For \(M,N\in\mathbb N\), consider the set
\[
\Omega_{M,N}
=
\left\{
f\in L^1(\mathbb T): S_N(f)\leq M
\right\}.
\]
Then
\[
L^1(\mathbb T)\setminus \mathcal Z_2
=
\bigcup_{M}\bigcap_{N}\Omega_{M,N}.
\]
Moreover, for each \(M,N\in\mathbb N\), the set \(\Omega_{M,N}\) is closed.
Indeed, let \(f_j\to f\) in \(L^1(\mathbb T)\), with \(f_j\in\Omega_{M,N}\) for
all \(j\). Since convergence in \(L^1(\mathbb T)\) implies convergence of the
Fourier coefficients, we have $\widehat{f_j}(n)\to \widehat f(n),$  $n\in\mathbb Z$. Hence, for fixed \(N\),
\[
S_N(f)
=
\lim_{j\to\infty} S_N(f_j)
\leq M.
\]
Thus \(f\in\Omega_{M,N}\), and so \(\Omega_{M,N}\) is closed and 
\[
\bigcap_{N}\Omega_{M,N}
\]
 is closed for every \(M\in\mathbb N\). Therefore
\(L^1(\mathbb T)\setminus \mathcal Z_2\) is an \(F_\sigma\) subset of $L^1(\mathbb{T})$.

It remains to prove that $\mathcal Z_2$ is dense. By Lemma~\ref{Duren lacunary example}, there exists
$f_0\in L^1(\mathbb T)$ such that
\[
\sum_{k=1}^{\infty}|\widehat f_0(2^k)|^2=+\infty.
\]
We claim that
\[
f_0+\mathcal P_{\mathrm{trig}}\subset\mathcal Z_2.
\]
Indeed, let $P\in\mathcal P_{\mathrm{trig}}$. Then there exists $M\in\mathbb N$ such that $\widehat P(n)=0$ for $(|n|>M).$ Hence, for all sufficiently large $k$,
\[
\widehat{f_0+P}(2^k)
=
\widehat f_0(2^k)+\widehat P(2^k)
=
\widehat f_0(2^k).
\]
Therefore
\[
\sum_{k=1}^{\infty}
|\widehat{f_0+P}(2^k)|^2
=
+\infty.
\]

Since $\mathcal P_{\mathrm{trig}}$ is dense in $L^1(\mathbb T)$, the set $f_0+\mathcal P_{\mathrm{trig}}$ is dense in $L^1(\mathbb T)$. Hence $\mathcal Z_2$ is dense in $L^1(\mathbb T)$.

\end{proof}

\begin{proof}[Proof of Theorem~\ref{alg spac criter}]
We prove that the complement of \(\mathcal A_p(X)\) is an \(F_\sigma\) set.
First, suppose that \(0<p<\infty\). We have
\[
X\setminus \mathcal A_p(X)
=
\left\{
f\in X:\beta(f)\in \ell^p
\right\}.
\]
For \(M,N\in\mathbb N\), set
\[
E_{M,N}
:=
\left\{
f\in X:
\sum_{n=0}^{N}|\beta_n(f)|^p\leq M
\right\}.
\]
Each \(E_{M,N}\) is closed. Indeed, let \(f_j\to f\) in \(X\), with
\(f_j\in E_{M,N}\) for every \(j\). Since convergence in \(X\) implies uniform
convergence on compact subsets of \(\mathbb D\), it also implies coefficientwise
convergence, that is, $\beta_n(f_j)\to \beta_n(f),$  $j\to\infty,$ for every \(n\in\mathbb N_0\). Therefore,
\[
\sum_{n=0}^{N}|\beta_n(f)|^p
=
\lim_{j\to\infty}\sum_{n=0}^{N}|\beta_n(f_j)|^p
\leq M.
\]
Hence \(f\in E_{M,N}\), and so \(E_{M,N}\) is closed.

Consequently,
\[
\bigcap_{N=1}^{\infty}E_{M,N}
\]
is closed for every \(M\in\mathbb N\). Since
\[
X\setminus \mathcal A_p(X)
=
\bigcup_{M=1}^{\infty}\bigcap_{N=1}^{\infty}E_{M,N},
\]
we conclude that \(X\setminus \mathcal A_p(X)\) is an \(F_\sigma\) set, and
therefore \(\mathcal A_p(X)\) is a \(G_\delta\) set.

For \(p=\infty\), we argue similarly.

It remains to prove density. Assume that $\mathcal A_p(X)\neq\varnothing,$
and choose $f\in\mathcal A_p(X).$
Let \(\mathcal P\) denote the space of polynomials. Since polynomials are dense in
\(X\), it is enough to show that
\[
f+\mathcal P\subset\mathcal A_p(X).
\]
Let \(P\in\mathcal P\). Since \(P\) has only finitely many non-zero Taylor
coefficients, the sequences $\beta(f+P)$ and  $\beta(f)$ differ only in finitely many terms. Therefore $\beta(f+P)\notin\ell^p.$

Hence
\[
f+\mathcal P\subset\mathcal A_p(X).
\]
Since \(f+\mathcal P\) is dense in \(X\), as a translation of a dense set, it follows that \(\mathcal A_p(X)\) is
dense in \(X\).

We now proceed to show algebraic genericity/spaceability.

Let
\[
f(z)=\sum_{n=0}^{\infty}a_nz^n\in X
\]
be such that \(\beta(f)\notin\ell^p\). Consider
\[
f_e(z)=\frac{f(z)+f(-z)}2=\sum_{m=0}^{\infty}a_{2m}z^{2m},
\]
and
\[
f_o(z)=\frac{f(z)-f(-z)}2=\sum_{m=0}^{\infty}a_{2m+1}z^{2m+1}.
\]
By assumption, \(f_e,f_o\in X\). Since \((a_n)\notin\ell^p\), at least one of
\((a_{2m})\) and \((a_{2m+1})\) is not in \(\ell^p\), where for \(p=\infty\) this means unbounded.

If \((a_{2m+1})\notin\ell^p\), set $h=f_o.$
If \((a_{2m})\notin\ell^p\), set $h(z)=zf_e(z).$
In both cases,
\[
h(z)=\sum_{m=0}^{\infty}b_mz^{2m+1}\in X
\]
and $(b_m)\notin\ell^p.$

For \(j\in\mathbb N_0\), define $g_j(z)=h(z^{2^j}).$ Then \(g_j\in X\), again from the assumptions, and
\[
g_j(z)=\sum_{m=0}^{\infty}b_mz^{2^j(2m+1)}.
\]
Thus the supports of the Taylor coefficients of \(g_j\) are contained in
\[
S_j=\{2^j(2m+1):m=0,1,2,\ldots\}.
\]
The sets \(S_j\) are pairwise disjoint.  Indeed $2^{j_1}(2m+1)\neq 2^{j_2}(2k+1)$ for every $j_1\neq j_2$.

We first prove algebraic genericity. Let \(d\) be a translation-invariant metric inducing the topology of \(X\). Let \((P_j)_{j\geq0}\) be an enumeration of all
polynomials with coefficients in \(\mathbb Q+i\mathbb Q\). Then every tail \(\{P_j:j\geq j_0\}\) is dense
in \(X\). Indeed, $X$ has no isolated points since it is a non-trivial topological vector space.

 Since scalar
multiplication is continuous, choose \(c_j\neq0\) such that $d(c_jg_j,0)<\frac1{j+1}.$
Set $F_j=P_j+c_jg_j.$ Then \(\{F_j:j\geq0\}\) is dense in \(X\). Indeed, if \(u\in X\) and \(\varepsilon>0\), choose
\(j_0\) such that \(d(c_jg_j,0)<\varepsilon/2\) for \(j\geq j_0\). Since the tail
\(\{P_j:j\geq j_0\}\) is dense, choose \(j\geq j_0\) with $d(P_j,u)<\frac{\varepsilon}{2}.$
Then
\[
d(F_j,u)\leq d(P_j,u)+d(c_jg_j,0)<\varepsilon.
\]

Let $F=\operatorname{span}\{F_j:j\geq0\}.$ Then \(F\) is dense in \(X\). We show that
\[
F\setminus\{0\}\subset\mathcal A_p(X).
\]

Let \(0\neq L\in F\). Then
\[
L=\sum_{j=0}^{N}\lambda_jF_j
\]
for some scalars \(\lambda_0,\ldots,\lambda_N\), not all zero. Hence
\[
L=P+\sum_{j=0}^{N}\lambda_jc_jg_j,
\]
where
\[
P=\sum_{j=0}^{N}\lambda_jP_j
\]
is a polynomial. Choose \(j_0\) such that $\lambda_{j_0}c_{j_0}\neq0.$ Since \(P\) has only finitely many non-zero Taylor coefficients and the supports
\(S_j\) are pairwise disjoint, for all sufficiently large \(m\),
\[
\beta_{2^{j_0}(2m+1)}(L)=\lambda_{j_0}c_{j_0}b_m.
\]
If \(0<p<\infty\), then
\[
\sum_m |\beta_{2^{j_0}(2m+1)}(L)|^p
=
|\lambda_{j_0}c_{j_0}|^p\sum_m |b_m|^p
=
+\infty.
\]
If \(p=\infty\), the sequence \((b_m)\) is unbounded, hence the sequence
\((\beta_{2^{j_0}(2m+1)}(L))_m\) is unbounded. Therefore
\[
\beta(L)\notin\ell^p.
\]
Thus \(L\in\mathcal A_p(X)\), and so \(\mathcal A_p(X)\) is algebraically generic.

We now prove spaceability. Let
\[
M=\overline{\operatorname{span}\{g_j:j\geq0\}}^{\,X}.
\]
Then \(M\) is a closed infinite-dimensional vector subspace of \(X\), since the \(g_j\)'s
are linearly independent.

Let \(0\neq u\in M\). Choose \(u_k\in\operatorname{span}\{g_j:j\geq0\}\) such that $u_k\to u$ in \(X\). Write
\[
u_k= \sum_{j=0}^{N_k}\lambda_{k,j}g_j.
\]
Since \(X\) is continuously embedded in \(H(\mathbb D)\), convergence in \(X\) implies
coefficientwise convergence. Since \(u\neq0\), there exists \(\ell_0\geq1\) such that $\beta_{\ell_0}(u)\neq0.$  We write uniquely $\ell_0=2^{j_0}(2m_0+1).$
 
Then
\[
\beta_{2^{j_0}(2m_0+1)}(u_k)=\lambda_{k,j_0}b_{m_0}.
\]
Passing to the limit gives \(b_{m_0}\neq0\) and $\lambda_{k,j_0}\to\lambda$ for some \(\lambda\neq0\). Therefore, for every \(m\geq0\),  
\[
\beta_{2^{j_0}(2m+1)}(u)
=
\lambda b_m.
\]
If \(0<p<\infty\), then
\[
\sum_m |\beta_{2^{j_0}(2m+1)}(u)|^p
=
|\lambda|^p\sum_m |b_m|^p
=
+\infty.
\]
If \(p=\infty\), then \((\beta_{2^{j_0}(2m+1)}(u))_m=(\lambda b_m)_m\) is unbounded.
Thus $\beta(u)\notin\ell^p.$
Hence
\[
M\setminus\{0\}\subset\mathcal A_p(X).
\]
Therefore \(\mathcal A_p(X)\) is spaceable.
\end{proof}

\begin{proof}[Proof of Corollary~\ref{cor:A-infty-A1}]
We apply the abstract coefficient criterion.

First let $X=\bigcap_{0<p<1}H^p.$
Then \(X\) is a metrizable topological vector space of holomorphic functions on
\(\mathbb D\), and convergence in \(X\) implies uniform convergence on compact
subsets of \(\mathbb D\). Moreover, polynomials are dense in \(X\) (see section~\ref{sec:2}). Also, if
\(f\in X\), then $f(-z), zf(z),f(z^N)$, 
belong to \(X\). Indeed,  a direct computation shows
\[
\sup_{0<r<1}\int_{0}^{2\pi}|re^{i\theta}f(re^{i\theta})|^p \, d\theta \leq \sup_{0<r<1}\int_{0}^{2\pi}|f(re^{i\theta})|^p<\infty,
\]
for every $0<p<1$.
The first and last claim follows directly, from  Littlewood's Subordination Theorem~\ref{subord}.

Now let $X=A(\mathbb D).$
Then \(A(\mathbb D)\) is a Banach space of holomorphic functions on \(\mathbb D\),
and convergence in \(A(\mathbb D)\) is uniform convergence on
\(\overline{\mathbb D}\), hence also on compact subsets of \(\mathbb D\). Polynomials
are dense in \(A(\mathbb D)\) (again see section~\ref{sec:2}). Moreover, if \(f\in A(\mathbb D)\), then  $f(-z), zf(z),f(z^N)$, 
belong to \(A(\mathbb D)\).

Finally, in both cases, the corresponding exceptional sets are nonempty; see
\cite{pandis2024some,duren1970theory}. Therefore, the abstract criterion applies
and yields the desired conclusion.
\end{proof}

\begin{proof}[Proof of Theorem~\ref{C_n-C_n-1}]
Fix $n\in\mathbb N$. Choose $0<\alpha<1$
and consider
\[
\varphi(z):=\frac{1}{(1-z)^\alpha}=\sum_{k=0}^{\infty}a_kz^k,\qquad z\in\mathbb D.
\]
Since $\alpha<1$, we have $\varphi\in H^p$ for every $0<p<1$ and thus $\varphi\in \bigcap_{0<p<1}H^p.$

Then
\[
a_k=\binom{k+\alpha-1}{k}\approx \frac{k^{\alpha-1}}{\Gamma(\alpha)}\, \quad (k\to \infty).
\]
(see \cite{NestoridisGenericH1,zygmund2002trigonometric}).

In particular, the sequence $(a_k)$ is bounded, while the sequence $(ka_k)$ is
unbounded.

For every $j\in\mathbb N_0$, define
\[
\varphi_j(z):=\frac{\varphi(z^{2^j})-\varphi(-z^{2^j})}{2}.
\]
Then $\varphi_j\in \bigcap_{0<p<1}H^p$, as described in Corollary~\ref{cor:A-infty-A1}, for every $j\in\mathbb N_0$. Moreover,
\[
\varphi_j(z)
=
\sum_{m=0}^{\infty}a_{2m+1}z^{2^j(2m+1)}.
\]
Thus the Taylor coefficient support of $\varphi_j$ are the pairwise disjoint sets
\[
S_j:=\{2^j(2m+1):m=0,1,2,\ldots\}.
\]

Let $F$ denote the primitive operator
\[
Fh(z):=\int_0^zh(\zeta)\,d\zeta.
\]
For every $j\in\mathbb N_0$, set
\[
u_j:=F^{\,n-1}\varphi_j.
\]
Then $u_j\in \bigcap_{0<p<1}H^p$, from Theorem~\ref{Nest-Thirios primitive}.
Moreover,
\[
u_j^{(n-1)}=\varphi_j,
\qquad
u_j^{(n)}=\varphi_j'.
\]

We first observe that $u_j\in\mathcal D_n$
for every $j\in\mathbb N_0$.

We now build a dense vector space contained, except for zero, in $\mathcal D_n$. Let $\{P_j\}_{j\in\mathbb N_0}$ be an enumeration of the polynomials with
coefficients in $\mathbb Q+i\mathbb Q$. Since polynomials are dense in $\bigcap_{0<p<1}H^p$, every tail $\{P_j:j\geq j_0\}$ is dense in $\bigcap_{0<p<1}H^p$. Since scalar multiplication is continuous in $\bigcap_{0<p<1}H^p$, for every $j\in\mathbb N_0$ we may
choose $c_j\in\mathbb C\setminus\{0\}$ such that $d(c_ju_j,0)<\frac{1}{j+1}.$
 
Set $F_j:=P_j+c_ju_j.$ The set $\{F_j:j\in\mathbb N_0\}$
is dense in $\bigcap_{0<p<1}H^p$. Let
\[
F:=\operatorname{span}\{F_j:j\in\mathbb N_0\}.
\]
Then $F$ is a dense vector subspace of $\bigcap_{0<p<1}H^p$. We shall prove that
\[
F\setminus\{0\}\subset \mathcal D_n.
\]

Let $0\neq L\in F$. Then there exist $N\in\mathbb N$ and scalars
$\lambda_0,\ldots,\lambda_N\in\mathbb C$, not all zero, such that
\[
L=\sum_{j=0}^{N}\lambda_jF_j.
\]
Thus
\[
L
=
\sum_{j=0}^{N}\lambda_jP_j
+
\sum_{j=0}^{N}\lambda_jc_ju_j.
\]
Choose $j_0\in\{0,\ldots,N\}$ such that $\lambda_{j_0}c_{j_0}\neq0.$

First, we show that $\beta(L^{(n-1)})\in\ell^\infty.$
Indeed,
\[
L^{(n-1)}
=
P+\sum_{j=0}^{N}\lambda_jc_j\varphi_j,
\]
where
\[
P:=\left(\sum_{j=0}^{N}\lambda_jP_j\right)^{(n-1)}
\]
is a polynomial. Each $\varphi_j$ has bounded Taylor coefficients, and only finitely
many $\varphi_j$ occur in the sum. Hence $\beta(L^{(n-1)})\in\ell^\infty.$
Therefore $L\notin\mathcal C_{n-1}.$

It remains to prove that $\beta(L^{(n)})\notin\ell^\infty.$
We have
\[
L^{(n)}
=
P'
+
\sum_{j=0}^{N}\lambda_jc_j\varphi_j'.
\]
The polynomial $P'$ has only finitely many non-zero Taylor coefficients.

Moreover, the supports of the functions $\varphi_j'$ are
\[
S_j-1:=
\{2^j(2m+1)-1:m=0,1,2,\ldots\}.
\]
which are still pairwise disjoint.  Therefore, on the support $S_{j_0}-1$, no cancellation with the other
$\varphi_j'$ can occur. Hence, for all sufficiently large $m$,
\[
\beta_{2^{j_0}(2m+1)-1}(L^{(n)})
=
\lambda_{j_0}c_{j_0}
2^{j_0}(2m+1)a_{2m+1}.
\]
Since $(2m+1)a_{2m+1}\to+\infty$ and $\lambda_{j_0}c_{j_0}\neq0,$
the sequence $\beta(L^{(n)})$ is unbounded. Thus $L\in\mathcal C_n.$

Combining the two conclusions, we get
\[
L\in\mathcal C_n\setminus\mathcal C_{n-1}
=
\mathcal D_n.
\]
Therefore the proof is complete.
\end{proof}

\begin{proof}[Proof of Theorem~\ref{alg spac nestor}]

The set \(\Lambda\) is dense \(G_\delta\) in \(H^1\). In particular, it is non-empty (\cite{NestoridisGenericH1}).  Choose $ f\in\Lambda$  and write \[ F(z):=\int_0^z f(\zeta)\,d\zeta = \sum_{n=0}^{\infty}a_nz^n. \] 

Since \(F(0)=0\), we have \(a_0=0\). By the choice of \(f\), $ (a_n)_{n\ge0}\notin\ell^p $ for every \(0<p<1\). We claim that either 
\[ (a_{2m})_{m\ge1}\notin\ell^p \quad\text{for every }0<p<1, \] 
or 
\[ (a_{2m+1})_{m\ge0}\notin\ell^p \quad\text{for every }0<p<1. \] 
Indeed, suppose not. Then there exist \(r,s\in(0,1)\) such that $ (a_{2m})_{m\ge1}\in\ell^r $ and $ (a_{2m+1})_{m\ge0}\in\ell^s.$  Let  $t:=\max\{r,s\}<1$.  Since \(\ell^r\subset\ell^t\) and \(\ell^s\subset\ell^t\), both subsequences belong to \(\ell^t\). Hence $ (a_n)_{n\ge0}\in\ell^t,$ which contradicts \(f\in\Lambda\). This proves the claim.

Now consider the odd and even parts of \(F\):
\[ F_o(z):=\frac{F(z)-F(-z)}{2} = \sum_{m=0}^{\infty}a_{2m+1}z^{2m+1}, 
\] 
and
\[ F_e(z):=\frac{F(z)+F(-z)}{2} = \sum_{m=1}^{\infty}a_{2m}z^{2m}. \] 

If $ (a_{2m+1})_{m\ge0}\notin\ell^p $ for every $0<p<1$,  set $ H:=F_o.$ 

Then \[ H(z)=\sum_{m=0}^{\infty}b_mz^{2m+1} \] with  $(b_m)_{m\ge0}\notin\ell^p$ for every $0<p<1.$ 

Moreover, \[ H'(z)=\frac{f(z)+f(-z)}{2}\in H^1. \] If instead $ (a_{2m})_{m\ge1}\notin\ell^p $ for every $ 0<p<1,$  set  $H(z):=zF_e(z).$
  
Again we may write, 
\[
H(z)=\sum_{m=0}^{\infty}b_mz^{2m+1} 
\] 
with  $(b_m)_{m\ge0}\notin\ell^p$ for every $0<p<1.$  For convenience we use the same notation for $(b_n)$.

Furthermore,
\[ H'(z)=F_e(z)+zF_e'(z).
\] 
Since \(F\) is the primitive of an \(H^1\)-function, we get \(F\in H^{\infty}\subset H^1\) from Lemma~\ref{f' H^1 f H^{infty}}, and hence $ F_e\in H^1. $

Also 
\[ 
F_e'(z)=\frac{f(z)-f(-z)}{2}\in H^1.
\] 
Thus $ H'\in H^1.$ In either case, we have constructed a holomorphic function 
\[ 
H(z)=\sum_{m=0}^{\infty}b_mz^{2m+1} 
\] 
such that $ h:=H'\in H^1 $
and $ (b_m)_{m\ge0}\notin\ell^p$ for every $0<p<1.$ For every \(j\in\mathbb N_0\), define
\[ 
H_j(z):=H(z^{2^j}) \quad \text{and}\quad g_j(z):=H_j'(z). 
\] 
Then 
\[ 
g_j(z)=2^jz^{2^j-1}h(z^{2^j}). 
\] 
Since \(h\in H^1\), it follows that \(g_j\in H^1\). Indeed, for fixed \(j\), using Theorem~\ref{subord} 
\[ 
\|g_j\|_{H^1} \leq 2^j\|h\|_{H^1} <+\infty.
\] 
Therefore  the primitive coefficient support of \(g_j\) is 
\[ 
S_j:=\{2^j(2m+1):m=0,1,2,\ldots\}. 
\] 
The sets \(S_j\), \(j\in\mathbb N_0\), are pairwise disjoint. We now prove algebraic genericity. Let \((P_j)_{j\ge0}\) be an enumeration of the polynomials with coefficients in \(\mathbb Q+i\mathbb Q\). Since polynomials are dense in \(H^1\), every tail $ \{P_j:j\ge j_0\} $ is dense in \(H^1\). 

Choose \(c_j\in\mathbb C\setminus\{0\}\) such that   $\|c_jg_j\|_{H^1}<\frac1{j+1}$ and  set   $G_j:=P_j+c_jg_j$.    Then the set \(\{G_j:j\ge0\}\) is dense in \(H^1\). Indeed, if \(u\in H^1\) and \(\varepsilon>0\), choose \(j_0\) such that   $\|c_jg_j\|_{H^1}<\frac{\varepsilon}{2}$   for every \(j\ge j_0\). 

The tail \(\{P_j:j\ge j_0\}\) is dense in \(H^1\), since $H^1$ has no isolated points, so there exists \(j\ge j_0\) such that   $\|P_j-u\|_{H^1}<\frac{\varepsilon}{2}$.   Hence \[ \|G_j-u\|_{H^1} \leq \|P_j-u\|_{H^1}+\|c_jg_j\|_{H^1} <\varepsilon. \]

Let \[ F:=\operatorname{span}\{G_j:j\ge0\}. \] Then \(F\) is a dense vector subspace of \(H^1\). We show that \[ F\setminus\{0\}\subset \Lambda. \] Let \(0\neq L\in F\). Then \[ L=\sum_{j=0}^{N}\lambda_jG_j \] for some scalars \(\lambda_0,\ldots,\lambda_N \in \mathbb{C}\), not all zero. Hence \[ L=P+\sum_{j=0}^{N}\lambda_jc_jg_j, \] where \[ P:=\sum_{j=0}^{N}\lambda_jP_j \] is a polynomial. Choose \(j_0\) such that $ \lambda_{j_0}c_{j_0}\neq0. $ 

Since \(F(P)\) is a polynomial, it has only finitely many non-zero Taylor coefficients. Since the supports \(S_j\) are pairwise disjoint, no cancellation occurs on \(S_{j_0}\). Therefore, for all sufficiently large \(m\), \[ a_{2^{j_0}(2m+1)}(L) = \lambda_{j_0}c_{j_0}b_m. \]

Since  $(b_m)\notin\ell^p$  for every \(0<p<1\), it follows that   $a(L)\notin\ell^p$   for every \(0<p<1\). Hence $ L\in\Lambda.$  This proves algebraic genericity.

We now prove spaceability. Let \[ M:=\overline{\operatorname{span}\{g_j:j\ge0\}}^{\,H^1}. \] Then \(M\) is a closed vector subspace of \(H^1\). It is infinite-dimensional because the functions \(g_j\) are linearly independent. Let \(0\neq u\in M\). Choose a sequence \[ u_k\in\operatorname{span}\{g_j:j\ge0\} \] such that  $ u_k\to u $ in \(H^1\). Write \[ u_k=\sum_{j=0}^{N_k}\lambda_{k,j}g_j. \] Convergence in \(H^1\) implies uniform convergence on compact subsets of \(\mathbb D\), hence coefficientwise convergence, i.e $\beta_{n}(u_k)\to \beta_n(u)$ as $k\to \infty$ for every $n$. This also leads to $a_{n}(u_k)=\frac{\beta_{n-1}(u_k)}{n}\to \frac{\beta_{n-1}(u)}{n}=a_n(u)$ as $k\to \infty$ for every $n$.

Since \(u\neq0\), the primitive \(F(u)\) is not identically zero. Hence there exists $ \ell_0\ge1$  such that $ a_{\ell_0}(u)\neq0.$   Write uniquely  $\ell_0=2^{j_0}(2m_0+1)$.   Then \[ a_{2^{j_0}(2m_0+1)}(u_k) = \lambda_{k,j_0}b_{m_0}. \] Passing to the limit gives \(b_{m_0}\neq0\) and   $\lambda_{k,j_0}\to\lambda$   for some \(\lambda\neq0\). Therefore, for every \(m\ge0\), \[ a_{2^{j_0}(2m+1)}(u) = \lambda b_m. \] Since \((b_m)\notin\ell^p\) for every \(0<p<1\), we obtain $ a(u)\notin\ell^p$   for every \(0<p<1\). Thus  $ u\in\Lambda.$   

Therefore \[ M\setminus\{0\}\subset\Lambda. \] Since \(M\) is closed and infinite-dimensional, \(\Lambda\) is spaceable in \(H^1\). \end{proof}

\noindent

{\bf Acknowledgements.} The author would like to thank Vassili Nestoridis for  helpful communication.

\bibliographystyle{amsplain}
\bibliography{bibli}

\end{document}